\documentclass[11pt,reqno]{amsart}

\usepackage{amsmath, amsfonts, amsthm, amssymb}

\numberwithin{equation}{section}

\newtheorem{Theorem}{Theorem}[section]
\newtheorem{Lemma}{Lemma}[section]

\newtheorem{Proposition}{Proposition}[section]

\theoremstyle{definition}

\newtheorem{Remark}{Remark}[section]

\renewcommand{\r}{\rho}

\def\i{\varepsilon}
\renewcommand{\u}{{\bf u}}

\newcommand{\R}{{\mathbb R}}
\newcommand{\Dv}{{\rm div}}

\newcommand{\dl}{\delta}

\def\f{\frac}

\def\D{\Delta }

\def\hf1{^\f{1}{1-\xi^2}}

\def\be{\begin{equation}}
\def\en{\end{equation}}
\def\bs{\begin{split}}
\def\es{\end{split}}

\newcommand{\F}{{\mathtt F}}

\title{Formation of singularity for compressible viscoelasticity} 

\author{Xianpeng Hu and Dehua Wang}
\address{Courant Institute of Mathematical Sciences, New York University, New York, NY 10012.}
\email{xianpeng@cims.nyu.edu}
\address{Department of Mathematics, University of Pittsburgh,
                           Pittsburgh, PA 15260, USA.}
\email{dwang@math.pitt.edu}

\keywords{Compressible viscoelastic fluid, inviscid elasticity, local classical
solution, formation of singularity, blowup, breakdown.}

\subjclass{35A05, 76A10, 76D03, 76P05, 82B40, 82C40.}
\date{}

\begin{document}

\begin{abstract}
The formation of singularity and breakdown of  classical solutions to the three-dimensional
 compressible viscoelasticity and inviscid elasticity are considered.
For the compressible inviscid elastic fluids, the finite-time formation of singularity in classical solutions is proved for certain initial data. 
For the compressible viscoelastic fluids, a criterion in term of the temporal integral of the velocity gradient is obtained for the breakdown of smooth solutions.
\end{abstract}

\maketitle

\section{Introduction}

We are concerned with the formation of singularities in smooth solutions to the multi-dimensional partial differential equations of viscoelasticity, especially, of the viscoelastic fluids.
Viscoelastic fluids exhibit a combination of both fluid and solid characteristics,
and keep memory of their past deformations due to their ``elastic" nature.  
Viscoelastic fluids have a wide range of applications  and hence have received
a great deal of interests. Examples and applications of
viscoelastic fluids include from oil, liquid polymers, 
to  bioactive fluids,  and viscoelastic blood flow  past valves; see
\cite{ KM} for more applications. For the viscoelastic
materials, the competition between the kinetic energy and the
internal elastic energy through the special transport properties
of their respective internal elastic variables makes the materials
more intractable in  understanding their behavior, since any
distortion of microstructures, patterns or configurations in the
dynamical flow will involve the deformation tensor. For classical
simple fluids,  the internal energy can be determined solely by
the determinant of the deformation tensor; however, the internal
energy of complex fluids carries all the information of the
deformation tensor. The interaction between the
microscopic elastic properties and the macroscopic fluid motions
leads to the rich and complicated rheological phenomena in
viscoelastic fluids, and also causes formidable analytic and
numerical challenges in mathematical analysis.
For the viscoelastic materials with significant viscosities, the equations of the compressible viscoelastic fluids of Oldroyd type (\cite{Oldroyd1, Oldroyd2}) in three spatial dimensions take the
following form (\cite{CD, Gurtin, Joseph}): 
\begin{equation}\label{cve}
\begin{cases}
\r_t+\Dv(\r\u)=0,\\
(\r\u)_t+\Dv(\r\u\otimes\u)+\nabla P=\mu\D\u+(\mu+\lambda)\nabla\Dv\u+\Dv(\r\F\F^\top),\\
\F_t+\u\cdot\nabla\F=\nabla\u\,\F,
\end{cases}
\end{equation}
where $\r$ stands for the density, $\u\in \R^3$ the velocity, and
$\F\in M^{3\times 3}$ (the set of $3\times 3$ matrices)  the
deformation gradient, and 
$$P=A\r^\gamma, \quad A>0, \; \gamma>1,$$  is the pressure. 
The viscosity coefficients $\mu, \lambda$
are two constants satisfying 
\begin{equation}\label{v2} 
\mu\ge 0,\quad 3\lambda+2\mu\ge 0,
\end{equation}
which ensure that the operator
$-\mu\D\u-(\lambda+\mu)\nabla\Dv\u$ is a strongly elliptic
operator.  The notation $\u\cdot\nabla\F$ is understood to be
$(\u\cdot\nabla)\F$ and $\F^\top$ means the transpose matrix of $\F$. 
As usual, we call the first equation in
\eqref{cve} the continuity equation. For system \eqref{cve}, the
corresponding elastic energy is chosen to be  the special form of
the Hookean linear elasticity:
$$W(\F)=\frac{1}{2}|\F|^2,$$
which, however, does not reduce the essential difficulties for
analysis. The methods and results of this paper can be applied to
more general cases. 
In the physical regime of negligible viscosities,  \eqref{cve}  becomes the system of inviscid  compressible elasticity:
\begin{equation}\label{1}
\begin{cases}
\r_t+\Dv(\r\u)=0,\\
(\r\u)_t+\Dv(\r\u\otimes\u)+\nabla P=\Dv(\r\F\F^\top),\\
\F_t+\u\cdot\nabla\F=\nabla\u\,\F.
\end{cases}
\end{equation}
We refer the readers to \cite{CD, Gurtin, Joseph, LW, RHN} for more discussions and physical background on viscoelasticity.
This paper is devoted to the study of  formation of  singularities for
both the inviscid elastic flow \eqref{1} and the viscoelastic flow \eqref{cve}.

Without the deformation gradient $\F$, the system \eqref{cve}
becomes the compressible Navier-Stokes equations. There is a huge
literature about solutions to the compressible Navier-Stokes
equations; see Danchin \cite{Danchin}, Feireisl \cite{Feireisl}, Lions \cite{Lions}, and references therein. 
In particular, the global weak solutions were
constructed in \cite{Feireisl, Lions} for $\gamma>\f{3}{2}$, while
in \cite{Danchin} the global existence of strong solution was proved in the critical spaces when the solution is a small perturbation of the equilibrium.

When the deformation gradient $\F$ does appear, the system \eqref{cve} is much more complicated than the Navier-Stokes equations, although the equation
satisfied by the deformation gradient $\F$ is a transport equation
and is similar to the continuity equation. Fortunately, \eqref{cve}
inhibits a list of local conservation laws which make the analysis
for the global existence of strong solutions available. 
The local strong solutions to the compressible viscoelastic flow \eqref{cve} with large data were obtained in Hu-Wang \cite{HW1, HW3}, and the global
strong solutions to the compressible viscoelastic flow \eqref{cve} with small data in the Besov spaces were established in Hu-Wang \cite{HW2, HW4} and Qiang-Zhang \cite{Zhang}. 
We remark that for these local and global existence of strong solutions, the local conservation laws are crucial for the dissipation of the deformation gradient, and the property that the curl of the deformation gradient is of higher order (Lemma 2.1 in \cite{HW1}) is also very helpful.
For large initial data, the global existence of strong or weak solutions for \eqref{cve} is still an outstanding open problem.  

Regarding the global existence of weak solutions to \eqref{cve} with large
data, among all of difficulties, the rapid oscillation of the density
and the non-compatibility between the quadratic form and the weak convergence are of the main issues; namely,  the local conservation laws are not enough for the convergence of the
nonlinear term, especially for the term related to the deformation
gradient although those terms sit well in the framework of
div-curl  structure, and the higher integrability of the density or the deformation gradient is not available up to now. 

For a strong solution to \eqref{cve} with large data, although the local existence was proved in \cite{HW1}, we expect that it will break down in a finite time as for the strong solution to the compressible Navier-Stokes equations (\cite{BKM, Ponce, HX}).  The breakdown of smooth solutions is due to the lack of  control
of the hydrodynamic variables, for instance, the $L^\infty$ norm of
the gradient of the velocity or the $L^\infty$ norm of the
density. As it is well-known, the $L^\infty$ norm of the gradient
of the velocity controls the $L^\infty$ norm of the density and
the deformation gradient in the compressible viscoelastic fluids \eqref{cve}. In this paper, we will provide a criterion and justify the blowup phenomena. Comparing with the compressible Navier-Stokes equations,  the viscoelastic flow \eqref{cve} is more complicated and thus more delicate and new estimates are needed for the analysis of strong solutions. More precisely,  the main difficulty lies in the estimates
of the gradients of the density and the deformation gradient, and  the
estimates  on the  $L^\infty_t H^1_x$ bounds of $\nabla\r$ and $\nabla\F$ are crucial. 
We find a criterion for breakdown of strong solutions of \eqref{cve} in term of the temporal integral of the $L^\infty$ norm of the velocity gradient.

For the inviscid flow \eqref{1},
similar to the compressible Euler equations (Sideris \cite{ST2}), 
we expect that the smooth solution to the system \eqref{1} will develop singularities in a finite time.
We will first reformulate the system \eqref{1} into a symmetric hyperbolic system so that a local smooth solution can be obtained from \cite{Kato, Majda}. Then we will prove that
the smooth solution cannot exist globally in time under some
restrictions on the initial data; that is,  the finite-time formation of singularities  is essentially inevitable
provided that the initial  velocity in some region near the
origin is supersonic relative to the sound speed at infinity. 
The proof will follow the idea of Sideris \cite{ST2} with more subtle estimates on the deformation gradients.

For the incompressible viscoelastic flows and related models, there are many papers in literature on classical solutions (cf. \cite{CM, CZ, KP, LLZH2,
LZP} and the references therein). On the other hand, the global existence of weak solutions
to the incompressible viscoelastic flows with large initial data
is also an outstanding open question, although there are some
progress in that direction (\cite{LLZH, LM, LW}).
For the inviscid elastodynamics, see \cite{ST} and
their references on the global existence of classical solutions.

The rest of this paper is organized as follows. In Section 2, we
will explain the mechanism which will ensure a local existence of smooth solution to
the inviscid compressible elastic fluid and also provide a proof of the finite-time
formation of singularities. In Section 3, we will
consider the  compressible viscoelastic fluids and prove a blowup criterion in term of the 
the $L^\infty$ norm of the gradient of the velocity.
\bigskip

\section{The Inviscid Case}

In this section, we consider the formation of singularities in smooth solutions of the inviscid flow \eqref{1}
in $\R^3$ with sufficiently smooth initial data:
\begin{equation}\label{in}
(\r, \u, \F)|_{t=0}=(\r_0(x), \u_0(x), \F_0(x)),\qquad x\in\R^3.
\end{equation}
We assume that
\begin{equation} 
\r_0(x)>0 \quad \textrm{for all}\quad x\in\R^3,
\end{equation}
 there exist positive constants $\bar{\r}_0$ and $R>0$ such that
\begin{equation}
(\r_0(x), \u_0(x), \F_0(x))=(\bar{\r}_0,0, I) \quad \textrm{for all}\quad |x|\ge R,
\end{equation}
where $I$ is the $3\times 3$ identity matrix, and
\begin{equation}\label{a}
\Dv(\r_0 \F^\top_0)=0,
\end{equation}

One useful property of the deformation gradient $\F$ is the following (see Lemma 6.1 in \cite{HW2}):
\begin{Lemma}\label{div}
If $(\r,\u,\F)$ is a smooth solution of \eqref{1}, and $\r,\F$
initially satisfy \eqref{a},
then the following identity  holds for any  time:
\begin{equation}\label{div1}
\Dv(\r\F^\top)=0.
\end{equation}
\end{Lemma}

Under the assumption \eqref{a} and using Lemma \ref{div}, we can rewrite
the system \eqref{1} for smooth solutions $(\r, \u, \F)$ with $\r>0$ as:
\begin{equation}\label{111}
\begin{cases}
\displaystyle \f{1}{\r c^2}\f{d \hat{P}}{dt}+\Dv\u=0,\\
\displaystyle \f{d\u}{dt}+\nabla \hat{P}=\F_{jk}\nabla_{x_j}\F_{ik}:=\sum_{k=1}^3(\F_k\cdot\nabla\F_k),\\
\displaystyle \f{d\F_{i}}{dt}=\nabla\u\,\F,
\end{cases}
\end{equation}
where 
$$\f{d}{dt}=\f{\partial}{\partial t}+\u\cdot\nabla,\quad
c^2=A\gamma(\gamma-1) \r^{\gamma-2}, \quad
\hat{P}=\f{A\gamma}{\gamma-1}\r^{\gamma-1},$$
 and $\F_i$ is the i-th column of the matrix $\F$.
Set 
$$
V=
\begin{bmatrix}
\hat{P}\\ \u\\  \F_1\\ \F_2\\ \F_3
\end{bmatrix}, \quad
A_0=
\begin{bmatrix}
\f{1}{\r c ^2}& 0\\
0&I_{12}
\end{bmatrix},
$$
and
$$
A_i=\begin{bmatrix}
\f{\u_i}{\r c^2}& e_i& 0& 0& 0\\
e_i^\top& \r\u_i I& -\F_{i1}I& -\F_{i2}I& -\F_{i3}I\\
0& -\F_{i1}I& \u_i I& 0& 0\\
0& -\F_{i2}I& 0&\u_i I& 0\\
0& -\F_{i3}I& 0& 0& \u_i I
\end{bmatrix}, \quad i=1, 2, 3,
$$
where $\{e_1,e_2, e_3\}$ is the standard basis of $\R^3$, $I_{12}$ is the $12\times 12$ identity matrix, $e_i^\top$ is the transpose of $e_i$, and $I$
is again the $3\times 3$ identity matrix. 
Then, in view of \eqref{111},  the system \eqref{1} can be written as a symmetric hyperbolic system of the form
\begin{equation}\label{sym}
A_0V_t+\sum_{i=1}^3A_i V_{x_i}=0,
\end{equation}
with $A_0(V)>0$ and $A_i(V)$ is a $13\times 13$ symmetric matrix for each
$i=1,2,3$. According to the well-known
result  in \cite{Kato, Majda},  the hyperbolic system of conservation laws \eqref{sym}
admits a local $C^1$ solution on some time interval $[0,T)$,
provided the initial data are sufficiently regular;  moreover,
$\r>0$ on $\R^3\times[0,T)$.

Since $\r_0(x)=\bar{\r}_0$,  $\u_0(x)=0$,  and $\F_0(x)=I$ for all $|x|\ge R$, then
$\sigma$, the sound speed at infinity, is given by
$$\sigma=\left(\f{\partial P}{\partial\r}(\bar{\r}_0)\right)^{\f{1}{2}}
=\left(A\gamma\bar{\r}_0^{\gamma-1}\right)^{\f{1}{2}}.$$ 
The following proposition is an immediate consequence of local energy
estimates (cf. \cite{ST1}). It simply states that the maximum
speed of propagation of the front of a smooth disturbances is
governed by $\sigma$.

\begin{Proposition}\label{p1}
If $(\r,\u, \F)\in C^1(\R^3\times[0,T))$ is a solution of
\eqref{1} and \eqref{in}, then 
$$(\r,\u, \F)=(\bar{\r}_0, 0, I) \text{ for all }
|x|\ge \sigma t+R \text{ and } 0\le t<T.$$
\end{Proposition}
 
As for the compressible Euler equations (Sideris \cite{ST2}), 
we expect that the smooth solution to the system \eqref{1} will develop singularities in a finite time.
The result in the following theorem shows that the $C^1$ solution to
\eqref{1} and \eqref{in} does not exist globally in time under some
restrictions on the initial data; more precisely, it states that
the finite-time formation of singularities in the three-dimensional
inviscid compressible elastic fluid is essentially inevitable
provided that the initial flow velocity, in some region near the
origin, is supersonic relative to the sound speed at infinity.

In order to state our result, we define
$$m(t)=\int_{\R^3}(\r(x,t)-\bar{\r}_0)dx,$$
$$\mathcal{F}(t)=\int_{\R^3}\r(x,t)x\cdot\u(x,t)dx,$$
$$\mathcal{E}(t)=\int_{\R^3}
\left(\f{1}{2}\r|\u|^2+\f{1}{2}\r|\F-I|^2+\f{P-P_0}{\gamma-1}\right)dx,$$
with
$$P_0=P(\bar\r_0)=A\bar\r_0^\gamma,$$
and
$$D(t)=\{x\in \R^3: |x|\le \sigma t+R\}.$$

\begin{Remark}\label{p1b}
Proposition \ref{p1} implies that the integrands of
 $m(t)$,  $\mathcal{F}(t)$, and $\mathcal{E}(t)$
 are identically equal to zero outside $D(t)$.
\end{Remark}

For two $3\times 3$ matrices $A$ and $B$, the following notations will be used:
$$A:B=\sum_{i,j=1}^3A_{ij}B_{ij},\quad|A|^2=\sum_{i,j=1}^3A_{ij}^2.$$

\begin{Lemma}\label{ei}
Let  $(\r,\u, \F)$ be a $C^1$ solution of \eqref{1} with initial data
\eqref{in}-\eqref{a}.
Then $\mathcal{E}(t)$ is conserved; that is,
\begin{equation}\label{EE}
\mathcal{E}'(t)=0,\qquad \mathcal{E}(t)=\mathcal{E}(0),
\end{equation}
 for all $t>0$.
\end{Lemma}

\begin{proof}
Multiplying the second equation in \eqref{1} by $\u$ and using the
first equation in \eqref{1} yield
\begin{equation}\label{222}
\begin{split}
\f{d}{dt}\int_{\R^3}\left(\f{1}{2}\r|\u|^2+\f{P-P_0}{\gamma-1}\right)dx
=-\int_{\R^3}\r\F\F^\top:\nabla\u dx.
\end{split}
\end{equation}
On the other hand, from the third and then the first equations of \eqref{1}, one deduces that
\begin{equation*}
\begin{split}
\f{\partial}{\partial t}\left(\r |\F|^2\right)&=\f{\partial \r}{\partial t}|\F|^2+2\r \F:\f{\partial \F}{\partial t}\\
&=\f{\partial \r}{\partial t}|\F|^2+2\r \F:(\nabla\u \,\F-\u\cdot\nabla \F)\\
&=\f{\partial \r}{\partial t}|\F|^2+2\r \F:(\nabla\u \,\F)-\r\u\cdot\nabla |\F|^2\\
&=\f{\partial \r}{\partial t}|\F|^2+2\r \F:(\nabla\u \,\F)+\Dv(\r\u)|\F|^2-\Dv(\r\u|\F|^2)\\
&=2\r \F:(\nabla\u \,\F)-\Dv(\r\u|\F|^2).
\end{split}
\end{equation*}
Since
\begin{equation*}
 \begin{split}
\r|\F-I|^2&=\r\sum_{i\neq j}\F_{ij}^2+\r\sum_{i=1}^3(\F_{ii}-1)^2\\
&=\r\sum_{i,j}\F_{ij}^2-2\r\sum_{i=1}^3\F_{ii}+3\r\\
&=\r|\F|^2-2\r\,\textrm{tr}(\F-I)-3\r,
 \end{split}
\end{equation*}
where $\textrm{tr}(\F-I)$ denotes the trace of the matrix $\F-I$, then
$$\r|\F|^2=\r |\F-I|^2+2\r\,\textrm{tr}(\F-I)+3(\r-\bar\r_0)+3\bar\r_0,$$
thus
$$\f{\partial}{\partial t}\left(\r |\F-I|^2+2\r\,\textrm{tr}(\F-I)+3(\r-\bar\r_0)\right)=2\r \F:(\nabla\u \,\F)-\Dv(\r\u|\F|^2).$$
Integrating the above equality, we arrive at
\begin{equation}\label{2222}
\begin{split}
&\f{1}{2}\f{d}{dt}\int_{\R^3}
\left(\r |\F-I|^2+2\r\,\textrm{tr}(\F-I)+3(\r-\bar\r_0)\right)dx\\
&=\int_{\R^3}\r \F:(\nabla\u\,\F)dx=\int_{\R^3}\r\F\F^\top:\nabla\u dx.
\end{split}
\end{equation}
Adding \eqref{222} and \eqref{2222} together, one has
\begin{equation}\label{22222}
\f{d}{dt}\int_{\R^3}
\left(\f{1}{2}\r|\u|^2+\f{1}{2}\r|\F-I|^2
+\r\,\textrm{tr}(\F-I)+\frac32(\r-\bar\r_0)
+\f{P-P_0}{\gamma-1}\right)dx=0.
\end{equation}
We note  from the  first equation of \eqref{1} that
\begin{equation}\label{TT1}
m'(t)=\frac{d}{dt}\int_{\R^3}\left(\r-\bar\r_0\right)dx
=-\int_{\R^3}\Dv (\r\u)dx=0,
\end{equation}
thus, 
\begin{equation}\label{TT2}
m(t)=m(0).
\end{equation}
Due to Lemma \ref{div}, we have, using integration by parts, 
$$\int_{\R^3}\r\F^\top:\nabla\u dx=0.$$ 
From the first and third equations in \eqref{1}, it is easy to deduce that
$$\partial_t(\r\,\textrm{tr}(\F-I))
+\Dv(\r\,\mathrm{tr}(\F-I)\,\u)=\r\F^\top:\nabla\u.$$
 Integrating the above equality yields
\begin{equation}\label{TT3}
\f{d}{dt}\int_{\R^3}\r\,\mathrm{tr}(\F-I)\,dx=0,
\end{equation}
and then 
\begin{equation}\label{TT4}
\int_{\R^3}\r\,\mathrm{tr}(\F-I)\,dx=\int_{\R^3}\r_0\,\mathrm{tr}(\F_0-I)\,dx.
\end{equation}
Substituting \eqref{TT1} and \eqref{TT3} into \eqref{22222}, we have
\begin{equation}\label{22222b}
\f{d}{dt}\int_{\R^3}
\left(\f{1}{2}\r|\u|^2+\f{1}{2}\r|\F-I|^2
+\f{P-P_0}{\gamma-1}\right)dx=0.
\end{equation}
Therefore, \eqref{EE} follows from \eqref{22222b}.
The proof is complete.
\end{proof}

Now we state and prove the finite-time formation of singularity for \eqref{1}.

\begin{Theorem}\label{T1}
Let $(\r,\u,  \F)\in C^1(\R^3\times[0,T))$ be a solution of
\eqref{1} with initial data
\eqref{in}-\eqref{a}. If 
\begin{gather}
m(0)\ge 0,\label{FF1}\\
\mathcal{F}(0)>\f{16\pi}{3}\sigma R^4\|\r_0\|_{L^\infty}, \label{FF}
\end{gather}
and
\begin{equation}\label{a2}
\int_{\R^3}\r_0\,\mathrm{tr}(I-\F_0)dx \ge 2\mathcal{E}(0),
\end{equation}
then $T$ is necessary finite.
\end{Theorem}

\begin{proof}
Let $(\r,\u,\F)$ be a $C^1$ solution of \eqref{1} and \eqref{in}.
Then, $\mathcal{F}\in C^1[0,T)$ and
$$\mathcal{F}'(t)=\int_{\R^3}x\cdot(\r_t\u+\r\u_t)dx.$$
It follows from the Proposition \ref{p1}, the first two equations
 in \eqref{1}, and integration by parts, that
\begin{equation}\label{5}
\begin{split}
\mathcal{F}'(t)=\int_{D(t)}\r|\u|^2dx+3\int_{D(t)}(P-P_0)dx+\int_{D(t)}x\cdot\Dv\left(\r\F\F^\top\right)dx.
\end{split}
\end{equation}

Using the H\"{o}lder inequality, \eqref{TT2} and \eqref{FF1}, we have,
\begin{equation*}
\begin{split}
\int_{D(t)}Pdx&=A\int_{D(t)}\r^\gamma dx\\
&\ge A(\textrm{vol}D(t))^{1-\gamma}\left(\int_{D(t)}\r
dx\right)^\gamma\\
&=A(\textrm{vol}D(t))^{1-\gamma}\left(m(0)+\textrm{vol}D(t)\bar\r_0\right)^\gamma\\
&\ge \textrm{vol}D(t)A\bar\r_0^\gamma =\int_{D(t)}P_0 dx,
\end{split}
\end{equation*}
thus
\begin{equation}\label{7}
\int_{D(t)}(P-P_0)dx\ge 0.
\end{equation}

Note that Lemma \ref{div} implies
$$\Dv(\r\F\F^\top)=\Dv(\r(\F-I)(\F-I)^\top)+\Dv(\r(\F-I)).$$
Then the divergence theorem yields, using Lemma \ref{ei},  \eqref{TT4},  \eqref{a2}  and \eqref{7},
\begin{equation}\label{6}
\begin{split}
&\int_{D(t)}x\cdot\Dv(\r\F\F^\top)dx\\
&=\int_{D(t)}x\cdot\Dv(\r(\F-I)(\F-I)^\top)dx+\int_{D(t)}x\cdot\Dv(\r(\F-I))dx\\
&=-\int_{D(t)}I:\r(\F-I)(\F-I)^\top dx-\int_{D(t)}I:\r(\F-I) dx\\
&=-\int_{D(t)}\r|\F-I|^2 dx-\int_{D(t)}\r\,\mathrm{tr}(\F-I) dx\\
&\ge-2\mathcal{E}(t)-\int_{D(t)}\r\,\mathrm{tr}(\F-I) dx\\
&=-2\mathcal{E}(0)+\int_{\R^3}\r_0\,\mathrm{tr}(I-\F_0)dx\ge 0.
\end{split}
\end{equation}

Therefore, \eqref{5}-\eqref{6} yield
\begin{equation}\label{8}
\mathcal{F}'(t)\ge\int_{D(t)}\r|\u|^2dx.
\end{equation}
On the other hand, \eqref{8} and the Cauchy-Schwarz inequality lead to
\begin{equation}\label{9}
\begin{split}
\mathcal{F}'(t)&\ge \mathcal{F}(t)^2\left(\int_{D(t)}|x|^2\r
dx\right)^{-1}\\
&\ge \mathcal{F}(t)^2(\sigma t+R)^{-2}\left(\int_{D(t)}\r
dx\right)^{-1}\\
&=\mathcal{F}(t)^2(\sigma
t+R)^{-2}\left(\int_{D(t)}\r_0(x)dx\right)^{-1}\\
&\ge\left(\f{4\pi}{3}(\sigma
t+R)^5\|\r_0\|_{L^\infty}\right)^{-1}\mathcal{F}(t)^2,
\end{split}
\end{equation}
where
$$\|\r_0\|_{L^\infty}=\sup_{x\in\R^3}\r_0(x).$$

Since $\mathcal{F}(0)>0$, then \eqref{9} implies that $\mathcal{F}(t)>0$ for
$0\le t<T$, and that
\begin{equation*}
\begin{split}
\mathcal{F}(0)^{-1}&\ge \mathcal{F}(0)^{-1}-\mathcal{F}(T)^{-1}\\
&\ge \f{3}{16\pi
\sigma\|\r_0\|_{L^\infty}}\left(\f{1}{R^4}-\f{1}{(\sigma
T+R)^4}\right),
\end{split}
\end{equation*}
which  shows that $T$ cannot  become arbitrarily
large without contradicting the assumption \eqref{FF}.
Therefore the proof is complete.
\end{proof}

\begin{Remark}
Taking $\r_0(x)=\bar{\rho}_0$,  we see that
\eqref{FF1}-\eqref{a2} can be satisfied if $\u_0$ is supersonic relative to the sound speed $\sigma$, 
the diagonal entries $[I-\F_0]_{ii}\in (0,1)$,  $i=1, 2, 3$, and $\bar\r_0$ is not very large.
\end{Remark}

\bigskip


\section{The Viscous Case}

This section is devoted to the study of formation of singularity and breakdown,  especially of the blowup criteria for the smooth solutions of compressible viscoelastic fluids \eqref{cve} in $\R^3$ with sufficiently smooth initial data:
\begin{equation}\label{in2}
(\r, \u, \F)|_{t=0}=(\r_0(x), \u_0(x), \F_0(x)),\qquad x\in\R^3.
\end{equation}
The  goal is to obtain a blowup criteria in
term of the $L^\infty$ norm of the velocity. In other words, the
$L^\infty$ norm of the velocity controls the blowup for the
 compressible viscoelastic fluids. For this purpose, we
need to assume that
\begin{equation}\label{va}
7\mu>\lambda.
\end{equation}
Obviously, this condition will be fulfilled physically if
viscosities $\mu$ and $\lambda$ satisfy the condition \eqref{v2}
and $\lambda\le 0$ simultaneously. 

Denote $D^k$ and $D^k_0$ as:
$$D^k(\R^3)
:=\{f\in L^1_{loc}(\R^3): \, \|\nabla^k
f\|_{L^2(\R^3)}<\infty\};$$ and
$$D_0^k(\R^3):=\{f\in L^6(\R^3): \, \|\nabla^k f\|_{L^2(\R^3)}<\infty\}.$$

Now, our blowup result for the system \eqref{cve} takes the following form:

\begin{Theorem}\label{vMT}
Assume that the initial data satisfy
\begin{equation}\label{v3}
0\le\r_0\in H^3(\R^3),\quad \u_0\in D_0^1(\R^3)\cap
D^3(\R^3),\quad \F_0\in H^3(\R^3),\quad\Dv(\r_0\F_0^\top)=0
\end{equation}
and
\begin{equation}\label{v4}
-\mu\D\u_0-(\lambda+\mu)\nabla\Dv\u_0+A\nabla\r_0^\gamma=\r_0 g
\end{equation}
for some  $g\in H^1(\R^3)$ with $\sqrt{\r_0}g\in L^2(\R^3)$.
Let $(\r,\u,\F)$ be a classical solutions to the system \eqref{cve}
satisfying
\begin{equation}\label{v5}
\begin{cases}
(\r, \F)\in C([0,T^*], H^3(\R^3)),\\
 \u\in C([0,T^*], D_0^1(\R^3)\cap D^3(\R^3))\cap L^2(0,T^*; D^4(\R^3)),\\
\u_t\in L^\infty(0,T^*; D_0^1(\R^3))\cap L^2(0, T^*; D^2(\R^3)),\\
\sqrt{\r}\u_t\in L^\infty(0,T^*; L^2(\R^3)),
\end{cases}
\end{equation}
and  $T^*$ be the maximal existence time. If $T^*<\infty$ and
\eqref{va} holds, then
\begin{equation}\label{v6}
\lim_{T\rightarrow T^*}\int_0^T\|\nabla\u\|_{L^\infty(\R^3)}dt=\infty.
\end{equation}
\end{Theorem}

The local existence up to a possible finite time $T^*$ can be
constructed in the functional framework \eqref{v5} with the
compatibility condition \eqref{v4} in spirit of the corresponding
results for the compressible Navier-Stokes equations (see for example \cite{CK}), 
thus we omit the details in this paper and focus on the  breakdown of smooth solutions.
To prove Theorem \ref{vMT}, the main difficulty lies in the estimates
of the gradients of the density and the deformation gradient
provided that the quantity in \eqref{v6} is finite. In fact, the
key estimates in our analysis is $L^\infty_t H^1_x$ bounds of
$\nabla\r$ and $\nabla\F$.

\subsection{Regularity of solutions}
In this subsection, we will derive a number of regularities of
solutions $(\r,\u,\F)$ to the system \eqref{cve} provided
\begin{equation}\label{v7}
\lim_{T\rightarrow
T^*}\int_0^T\|\nabla\u\|_{L^\infty(\R^3)}dt<\infty.
\end{equation}
The standard energy estimate gives
$$\sup_{0\le t\le T}\Big(\|\sqrt{\r}\u\|^2_{L^2}+\|\r\|_{L^\gamma}
+\|\sqrt{\r}\F\|_{L^2}^2\Big)+\int_0^T\|\nabla\u\|_{L^2}^2dt\le
C,\quad 0\le T\le T^*.$$ 
Moreover, from the two transport equations
for the density and the deformation gradient, the assumption
\eqref{v7} yields the $L^\infty$ bounds for the density and the
deformation gradient respectively.

The first step for the regularity of solutions is to improve the
integrability of the velocity. In fact, we have

\begin{Lemma}
Under the assumption \eqref{va}, there exists a small $\dl>0$ such
that
$$\sup_{0\le t\le T}\int_{\R^3}\r|\u|^{3+\dl}dx\le C, \quad
0<T<T^*,$$ where $C$ is a positive constant depending only on
$\|\r\|_{L^\infty}$ and $\|\F\|_{L^\infty}$.
\end{Lemma}

\begin{proof}
The argument is similar to that of \cite{HX}. 
For $q>3$, we multiply the second equation in
\eqref{cve} by $q|\u|^{q-2}\u$ and integrate over $\R^3$ to obtain, using  the conservation of mass, 
\begin{equation}\label{v8}
\begin{split}
&\f{d}{dt}\int_{\R^3}\r|\u|^qdx+\int_{\R^3}\Big(q|\u|^{q-2}\big(\mu|\nabla\u|^2+(\lambda+\mu)|\Dv\u|^2+\mu(q-2)|\nabla|\u||^2\big)\\
&\qquad\qquad\qquad\qquad\qquad+q(\lambda+\mu)(\nabla|\u|^{q-2})\cdot\u\Dv\u\Big)dx\\
& =Aq\int_{\R^3}\Dv(|\u|^{q-2}\u)\r^\gamma
dx-q\int_{\R^3}\r\F\F^\top:\nabla(|\u|^{q-2}\u) dx\\
& \le C\int_{\R^3}\sqrt{\r}|\u|^{q-2}|\nabla\u|dx\le
\i\int_{\R^3}|\u|^{q-2}|\nabla\u|^2dx+C\int_{\R^3}\r|\u|^{q-2}dx\\
& \le\i\int_{\R^3}|\u|^{q-2}|\nabla\u|^2dx+C\left(\int_{\R^3}\r|\u|^{q}dx\right)^{\f{q-2}{q}},
\end{split}
\end{equation}
for  $\i>0$. 
As in \cite{HX}, if $7\mu>\lambda$, one has
\begin{equation}\label{v9}
\begin{split}
&q|\u|^{q-2}\Big(\mu|\nabla\u|^2+(\lambda+\mu)|\Dv\u|^2+\mu(q-2)|\nabla|\u||^2\Big)\\
&\qquad+q(\lambda+\mu)(\nabla|\u|^{q-2})\cdot\u\,\Dv\u\\
&\ge C|\u|^{q-2}|\nabla\u|^2.
\end{split}
\end{equation}
Substituting \eqref{v9} into \eqref{v8}, and taking $\i$ small
enough, we get by Gronwall's inequality,
$$\sup_{0\le t\le T}\int_{\R^3}\r|\u|^{3+\dl}dx\le C, \quad
0<T<T^*,$$
for some small $\delta>0$.
\end{proof}

The second step for the regularity considers  the
integrability of the material derivative of the velocity, which is
useful for the bounds on $\nabla\r$. To this end, we have,

\begin{Lemma}
Let $$G=\r\u_t+\r\u\cdot\nabla\u.$$  Then, 
$$\int_0^T\!\!\int_{\R^3}G^2dxdt+\sup_{0\le t\le
T}\int_{\R^3}|\nabla\u|^2dx\le
C\int_0^T\!\!\int_{\R^3}(|\nabla\r|^2+|\nabla\F|^2)dxdt+C$$ for all
$0\le T\le T^*$.
\end{Lemma}

\begin{proof}
From the $L^\infty$ bound of $\r$, we have
\begin{equation}\label{v10}
\int_0^T\!\!\int_{\R^3}G^2dxdt\le
C\int_0^T\!\!\int_{\R^3}\r\u_t^2dxdt+2\int_0^T\!\!\int_{\R^3}|\r\u\cdot\nabla\u|^2dxdt.
\end{equation}
We proceed to estimate the right-hand side of \eqref{v10} term by
term as follows. For the last term of \eqref{v10}, we have
\begin{equation}\label{v11}
\begin{split}
&\int_0^T\!\!\int_{\R^3}\r^2|\u|^2|\nabla\u|^2dxdt \\
&\le C\int_0^T\|\sqrt{\r}\u\|_{L^4}^2\|\nabla\u\|_{L^4}^2dt\\
&\le
C\int_0^T\|\sqrt{\r}\u\|^\alpha_{L^{3+\alpha}}\|\sqrt{\r}\u\|_{L^6}^{2-\alpha}\|\nabla\u\|_{L^4}^2dt\\
&\le
C\int_0^T\|\nabla\u\|_{L^2}^{3-\alpha}\|\nabla\u\|_{L^\infty}dt\\
&\le C\sup_{0\le T\le
T^*}\|\nabla\u\|_{L^2}^{3-\alpha}\int_0^T\|\nabla\u\|_{L^\infty}dt\\
&\le C\sup_{0\le T\le T^*}\|\nabla\u\|_{L^2}^{3-\alpha},
\end{split}
\end{equation}
where 
$$\f{\alpha}{3+\dl}+\f{2-\alpha}{6}=\f12, \quad
1<\alpha=\f{3+\dl}{3-\dl}<2.$$

For the first term, we multiply the second equation of \eqref{cve}
by $u_t$ and integrate over $\R^3$ to obtain
\begin{equation}\label{v12}
\begin{split}
&\int_{\R^3}\r\u_t^2dx+\int_{\R^3}\r\u\cdot\nabla\u\cdot\u_t
dx+\f{d}{dt}\int_{\R^3}\Big(\f{\mu}{2}|\nabla\u|^2+\f{\lambda+\mu}{2}|\Dv\u|^2\Big)dx\\
&\quad=A\int_{\R^3}\r^\gamma\Dv\u_t
dx-\int_{\R^3}\r\F\F^\top:\nabla\u_tdx.
\end{split}
\end{equation}
From the first equation and the third equation in \eqref{cve}, we
deduce that
\begin{equation}\label{v13}
\begin{split}
&\int_{\R^3}\r^\gamma\Dv\u_t dx\\
&=\f{d}{dt}\int_{\R^3}\r^\gamma\Dv\u
dx-\gamma\int_{\R^3}\r_t\r^{\gamma-1}\Dv\u dx\\
&=\f{d}{dt}\int_{\R^3}\r^\gamma\Dv\u
dx+\gamma\int_{\R^3}\Dv(\r\u)\r^{\gamma-1}\Dv\u dx\\
&=\f{d}{dt}\int_{\R^3}\r^\gamma\Dv\u
dx-\int_{\R^3}\r^{\gamma}\u\cdot\nabla\Dv\u
dx+(\gamma-1)\int_{\R^3}\r^\gamma|\Dv\u|^2dx\\
&\le\f{d}{dt}\int_{\R^3}\r^\gamma\Dv\u
dx+C\|\sqrt{\r}\u\|_{L^2}\|\nabla^2\u\|_{L^2(\R^3)}+C\|\nabla\u\|_{L^2(\R^3)}^2\\
&\le\f{d}{dt}\int_{\R^3}\r^\gamma\Dv\u
dx+C+\i\|G\|_{L^2(\R^3}^2+C\|\nabla\r\|_{L^2(\R^3)}^2\\
&\quad+C\|\nabla\F\|_{L^2(\R^3)}^2+C\|\nabla\u\|_{L^2(\R^3)}^2
\end{split}
\end{equation}
and
\begin{equation}\label{v14}
\begin{split}
&\int_{\R^3}\r\F\F^\top:\nabla\u_tdx\\
&=\f{d}{dt}\int_{\R^3}\r\F\F^\top:\nabla\u
dx-\int_{\R^3}\partial_t(\r\F\F^\top):\nabla\u dx\\
&\le\f{d}{dt}\int_{\R^3}\r\F\F^\top:\nabla\u
dx-\int_{\R^3}\r\F_{ik}\F_{jk}\u\cdot\nabla\partial_j\u^i dx\\
&\quad +C\int_{\R^3}\r|\F|^2|\nabla\u|^2dx\\
&\le \f{d}{dt}\int_{\R^3}\r\F\F^\top:\nabla\u
dx+C\|\sqrt{\r}\u\|_{L^2}\|\nabla^2\u\|_{L^2(\R^3)}+C\|\nabla\u\|_{L^2(\R^3)}^2\\
&\le \f{d}{dt}\int_{\R^3}\r\F\F^\top:\nabla\u
dx+C+\i\|G\|_{L^2(\R^3}^2+C\|\nabla\r\|_{L^2(\R^3)}^2\\
&\quad+C\|\nabla\F\|_{L^2(\R^3)}^2+C\|\nabla\u\|_{L^2(\R^3)}^2,
\end{split}
\end{equation}
where we used
$$\partial_t(\r\F\F^\top)+\u\cdot\nabla(\r\F\F^\top)=\nabla\u(\r\F\F^\top)+\r\F\F^\top(\nabla\u)^\top-\r\F\F^\top\Dv\u.$$
Note that
\begin{equation*}
\begin{split}
&\left|\int_0^T\!\!\int_{\R^3}\r\u\cdot\nabla\u\cdot\u_tdxdt\right|\\
&\le \f12\int_0^T\!\!\int_{\R^3}\r\u_t^2dxdt+\int_0^T\!\!\int_{\R^3}\r|\u\cdot\nabla\u|^2dxdt\\
&\le\f12\int_0^T\!\!\int_{\R^3}\r\u_t^2dxdt+C\sup_{0\le
T<T^*}\|\nabla\u\|_{L^2}^{3-\alpha}.
\end{split}
\end{equation*}
Substituting \eqref{v13} and \eqref{v14} to \eqref{v12},
using Gronwall's inequality and the Cauchy-Schwarz inequality, we have,
$$\sup_{0\le t\le T^*}\|\nabla\u\|_{L^2(\R^3)}\le C;$$
and
\begin{equation}\label{v15}
\begin{split}
&\int_0^T\!\!\int_{\R^3}\r\u_t^2dxdt\\
&\le C\int_0^T\!\!\int_{\R^3}(|\nabla\r|^2+|\nabla\F|^2)dxdt+\i\int_0^T\!\!\int_{\R^3}G^2dxdt\\
&\quad +C\sup_{0\le T<T^*}\|\nabla\u\|_{L^2}^{3-\alpha}+C.
\end{split}
\end{equation}
Here we used the following estimates: for all $0\le t\le T$
$$\left|\int_{\R^3}\r\F\F^\top:\nabla\u dx\right|\le C\|\sqrt{\r}\F\|_{L^2}\|\nabla\u\|_{L^2}\le C+\f{\mu}{8}\|\nabla\u\|_{L^2}^2$$
and
$$\left|\int_{\R^3}\r^\gamma\Dv\u dx\right|\le C\|\r^\gamma\|_{L^1}\|\nabla\u\|_{L^2}\le C+\f{\mu}{8}\|\nabla\u\|_{L^2}^2.$$
Choosing a sufficiently small $\i$, from \eqref{v10} and
\eqref{v15}, one has
$$\int_0^T\!\!\int_{\R^3}G^2dxdt+\sup_{0\le t\le
T}\int_{\R^3}|\nabla\u|^2dx\le
C\int_0^T\!\!\int_{\R^3}(|\nabla\r|^2+|\nabla\F|^2)dxdt+C.$$ 
The proof is complete.
\end{proof}

With the aid of the estimate of the material derivative of the
velocity, we can obtain the $L^\infty(0,T; L^2)$ estimates of
$\nabla\r$ and $\nabla\F$ and we actually have

\begin{Lemma}
Under the assumption \eqref{va}, the following estimates hold for
$0\le T\le T^*$:
\begin{equation}\label{v16}
\sup_{0\le T\le T^*}\int_{\R^3}|\nabla\r|^2dx\le C,
\end{equation}
\begin{equation}\label{v17}
\sup_{0\le T\le T^*}\int_{\R^3}|\nabla\F|^2dx\le C,
\end{equation}
\begin{equation}\label{v18}
\int_0^T\!\!\int_{\R^3}\r\u_t^2dxdt+\sup_{0\le T\le
T^*}\int_{\R^3}|\nabla\u|^2dx+\int_0^T\|\u\|_{H^2(\R^3)}^2dt\le C,
\end{equation}
\end{Lemma}

\begin{proof}
The arguments of \eqref{v16}, \eqref{v18} are similar to
Proposition 2.4 in \cite{HX} provided \eqref{v17} holds, and thus
we omit them. For \eqref{v17}, we can proceed as follows.

Differentiating the third equation in \eqref{cve} with respect to
$x_i$ (denoting $\partial_{x_i}$ by $\partial_i$) and multiplying the resulting identity by $2\partial_{i}\F$
yield
\begin{equation*}
\partial_t|\partial_i\F|^2+\u\cdot\nabla|\partial_i\F|^2=2\nabla\partial_i\u\F:\partial_i\F+2\nabla\u\partial_i\F:\partial_i\F.
\end{equation*}
Integrating the above identity over $\R^3$, one has
\begin{equation}\label{v20}
\begin{split}
\f{d}{dt}\int_{\R^3}|\partial_i\F|^2dx&=\int_{\R^3}|\partial_i\F|^2\Dv\u
dx+2\int_{\R^3}\left(\nabla\partial_i\u\F:\partial_i\F+\nabla\u\partial_i\F:\partial_i\F\right)dx\\
&\le
C\|\nabla\u\|_{L^\infty}\|\partial_i\F\|_{L^2}^2+C\|\partial_i\F\|_{L^2}\|\nabla^2\u\|_{L^2}\\
&\le
C\|\nabla\u\|_{L^\infty}\|\partial_i\F\|_{L^2}^2+C\|\partial_i\F\|_{L^2}^2+C\|\nabla^2\u\|^2_{L^2}\\
&\le
C\|\nabla\u\|_{L^\infty}\|\partial_i\F\|_{L^2}^2+C\|\partial_i\F\|_{L^2}^2+C\|\partial_i\r\|_{L^2}^2+C\|G\|_{L^2}^2
\end{split}
\end{equation}
since
$$-\mu\D\u-(\lambda+\mu)\nabla\Dv\u=-G-A\nabla\r^\gamma-\Dv(\r\F\F^\top).$$
From  \eqref{v16}, \eqref{v20} and Gronwall's inequality,
we obtain \eqref{v17}. 
The proof is complete.
\end{proof}

The next step is to obtain the uniform in time estimate for $\u_t$ as the following:

\begin{Lemma}
Under the assumption \eqref{va}, the following estimates hold for
all $0\le T\le T^*$,
\begin{equation}\label{v21}
\sup_{0\le t\le
T}\|\sqrt{\r}\u_t\|_{L^2}^2+\int_0^T\!\!\int_{\R^3}|\nabla\u_t|^2dxdt\le C,
\end{equation}
\begin{equation}\label{v221}
\sup_{0\le T\le T^*}\|\u\|_{H^2}\le C.
\end{equation}
\end{Lemma}

\begin{proof}
The argument of \eqref{v221} is similar to Proposition 2.5 in
\cite{HX}, and thus we omit it.
To prove \eqref{v21}, we differentiating the second equation in
\eqref{cve} with respect to $t$ to obtain
\begin{equation}\label{v23}
\begin{split}
&\r\u_{tt}+\r\u\cdot\nabla\u_t-\mu\D\u_t-(\lambda+\mu)\nabla\Dv\u_t+A\nabla
(\r^\gamma)_t\\
&=-\r_t(\u_t+\u\cdot\nabla\u)-\r\u_t\cdot\nabla\u+\Dv(\r\F\F^\top)_t.
\end{split}
\end{equation}

Taking the inner product of the above equation with $\u_t$ in
$L^2(\R^3)$ and integrating by parts, one obtains,
\begin{equation}\label{v24}
\begin{split}
&\f12\f{d}{dt}\int_{\R^3}\r\u_t^2dx+\int_{\R^3}(\mu|\nabla\u_t|^2+(\lambda+\mu)|\Dv\u_t|^2)dx\\
&\quad-A\int_{\R^3}(\r^\gamma)_t\Dv\u_tdx+\int_{\R^3}(\r\F\F^\top)_t:\nabla\u_t\\
&=-\int_{\R^3}\left(\r\u\cdot\nabla\left[\left(\f12\u_t+\u\cdot\nabla\u\right)\u_t\right]+\r\u_t\cdot\nabla\u\cdot\u_t\right)dx.
\end{split}
\end{equation}

From the first and the third equations in \eqref{cve}, we have,
\begin{equation}\label{v25}
\begin{split}
&\int_{\R^3}(\r^\gamma)_t\Dv\u_t dx\\
&=-\int_{\R^3}\Dv(\r^\gamma\u)\Dv\u_tdx-(\gamma-1)\int_{\R^3}\r^\gamma\Dv\u\Dv\u_tdx\\
&=\int_{\R^3}(\r^\gamma\Dv\u+\nabla\r^\gamma\cdot\u)\Dv\u_tdx-(\gamma-1)\int_{\R^3}\r^\gamma\Dv\u\Dv\u_tdx\\
&\le C\|\nabla\u_t\|_{L^2}\|\nabla\u\|_{L^2}+C\|\u\|_{L^\infty}\|\nabla\r\|_{L^2}\|\nabla\u_t\|_{L^2}\\
&\le\i\|\nabla\u_t\|_{L^2}^2+C\|\u\|_{H^2}^2
\end{split}
\end{equation}
since
$$(\r^\gamma)_t+\Dv(\r^\gamma\u)+(\gamma-1)\r^\gamma\Dv\u=0;$$
and
\begin{equation}\label{v26}
\begin{split}
&\int_{\R^3}(\r\F\F^\top)_t:\nabla\u_tdx\\
&=-\int_{\R^3}\u\cdot\nabla(\r\F\F^\top):\nabla\u_tdx+\int_{\R^3}\nabla\u\r\F\F^\top:\nabla\u_tdx\\
&\quad+\int_{\R^3}\r\F\F^\top(\nabla\u)^\top:\nabla\u_tdx+\int_{\R^3}\r\F\F^\top\Dv\u:\nabla\u_tdx\\
&\le
C\|\u\|_{L^\infty}(\|\nabla\r\|_{L^2}+\|\nabla\F\|_{L^2})\|\nabla\u_t\|_{L^2}+C\|\nabla\u\|_{L^2}\|\nabla\u_t\|_{L^2}\\
&\le \i\|\nabla\u_t\|_{L^2}^2+C\|\u\|_{H^2}^2.
\end{split}
\end{equation}

For the right-hand side of \eqref{v24}, we have
\begin{equation}\label{v27}
\begin{split}
&\left|-\int_{\R^3}\left(\r\u\cdot\nabla\left[\left(\f12\u_t+\u\cdot\nabla\u\right)\u_t\right]+\r\u_t\cdot\nabla\u\cdot\u_t\right)dx\right|\\
&\quad\le
\int_{\R^3}\Big(\r|\u||\u_t||\nabla\u_t|+\r|\u||\u_t||\nabla\u|^2+\r|\u|^2|\u_t||\nabla^2\u|\\
&\qquad\qquad\qquad +\r|\u|^2|\nabla\u||\nabla\u_t|+\r|\u_t|^2|\nabla\u|\Big)dx\\
&\quad\le
C\Big(\|\u\|_{L^6}\|\sqrt{\r}\u_t\|_{L^3}\|\nabla\u_t\|_{L^2}+\|\u\|_{L^6}\|\u_t\|_{L^6}\|\nabla\u\|_{L^3}^2\\
&\qquad +\|\u^2\|_{L^3}\|\u_t\|_{L^6}\|\nabla^2\u\|_{L^2}
+\|\nabla\u_t\|_{L^2}\|\nabla\u\|_{L^6}\|\u^2\|_{L^3}\\
&\qquad  +\|\sqrt{\r}\u_t\|^2_{L^2}\|\nabla\u\|_{L^\infty}\Big)\\
&\quad\le
C\Big(\|\sqrt{\r}\u_t\|_{L^2}^{\f12}\|\nabla\u_t\|_{L^2}^{\f32}+\|\nabla\u_t\|_{L^2}\|\nabla\u\|_{L^6}+\|\u_t\|_{L^6}\|\nabla^2\u\|_{L^2}\\
&\qquad+\|\nabla\u\|_{L^6}\|\nabla\u_t\|_{L^2}+\|\sqrt{\r}\u_t\|^2_{L^2}\|\nabla\u\|_{L^\infty}\Big)\\
&\quad\le
\i\|\nabla\u_t\|^2_{L^2}+C(1+\|\nabla\u\|_{L^\infty})\|\sqrt{\r}\u_t\|_{L^2}^2+C\|\u\|_{H^2}^2.
\end{split}
\end{equation}

Substituting \eqref{v25}, \eqref{v26} and \eqref{v27} into
\eqref{v24}, we get
\begin{equation}\label{v28}
\begin{split}
&\f{d}{dt}\int_{\R^3}\f12\r\u_t^2dx+\mu\int_{\R^3}|\nabla\u_t|^2dx\\
& \le \i\int_{\R^3}|\nabla\u_t|^2dx+C((1+\|\nabla\u\|_{L^\infty})\|\sqrt{\r}\u_t\|_{L^2}^2+\|\u\|_{H^2}^2).
\end{split}
\end{equation}
Thanks to the compatibility condition, it holds
$$\sqrt{\r_0}\u_t(0,x)\in L^2(\R^3),$$
and thus, for sufficiently small $\i$, \eqref{v28} gives us
\eqref{v21}.
The proof is complete.

\end{proof}

Finally, we now can obtain bounds of the first order derivatives
of the density and the second derivatives of the velocity.

\begin{Lemma}
Under the assumption \eqref{va}, the following estimates hold for
$0\le T\le T^*$,
\begin{equation}\label{v29}
\sup_{0\le t\le T}(\|\r_t(t)\|_{L^6}+\|\r\|_{W^{1,6}})\le C,
\end{equation}
\begin{equation}\label{v30}
\sup_{0\le t\le T}(\|\F_t(t)\|_{L^6}+\|\F\|_{W^{1,6}})\le C,
\end{equation}
\begin{equation}\label{v31}
\int_0^T\|\u(t)\|_{W^{2,6}}^2dt\le C.
\end{equation}
\end{Lemma}

\begin{proof}
The arguments of \eqref{v29} and \eqref{v31} are similar to those
in \cite{HX} provided \eqref{v30} is verified, and hence are omitted here. 
To establish \eqref{v30}, we first deduce from the previous
estimates that,
$$\u_t\in L^2(0,T; L^6(\R^3)),\quad G\in L^2(0,T; L^6(\R^3)).$$
Differentiating the third equation in \eqref{cve} with respect to
$x_i$, and multiplying the resulting identity by
$6|\partial_i\F|^4\partial_i\F$, one has,
\begin{equation}\label{v32}
\begin{split}
&\f{d}{dt}\int_{\R^3}|\partial_i\F|^6dx\\&=-6\int_{\R^3}|\partial_i\F|^6\Dv\u
dx-6\int_{\R^3}|\partial_i\F|^4\partial_i\F\partial_i\nabla\u\F
dx-6\int_{\R^3}|\partial_i\F|^4\nabla\u\partial_i\F\partial_i\F
dx\\
&\le
C\Big(\|\nabla\u\|_{L^\infty}\|\nabla\F\|^6_{L^6}+\|\nabla\F\|_{L^6}^5(\|\nabla\r\|_{L^6}+\|\nabla\F\|_{L^6}+\|G\|_{L^6})\Big).
\end{split}
\end{equation}
The above inequality and Gronwall's inequality imply
$$\sup_{0\le t\le T}\|\nabla\F\|_{L^6}\le C.$$
This estimate, together with the third equation in \eqref{cve} give
the bound of $\F_t$ in \eqref{v30}.
The proof is complete.
\end{proof}

\subsection{Proof of Theorem \ref{vMT}} 
To begin with, we have the
following higher order estimates for the density and the
deformation gradient:

\begin{Lemma}\label{y1}
 For all $0\le T\le T^*$, the following estimates hold:
$$\|\r\|_{L^\infty(H^2)}+\|\r_t\|_{L^\infty(H^1)}+\|\r_{tt}\|_{L^2}\le
C,$$
$$\|\r^\gamma\|_{L^\infty(H^2)}+\|(\r^\gamma)_t\|_{L^\infty(H^1)}+\|(\r^\gamma)_{tt}\|_{L^2}\le
C,$$ and
$$\|\F\|_{L^\infty(H^2)}+\|\F_t\|_{L^\infty(H^1)}+\|\F_{tt}\|_{L^2}\le
C.$$
\end{Lemma}

\begin{proof}
The argument for $\r$ and $\r^\gamma$ is similar to those in
\cite{HX}, and thus we will focus on the estimates on $\F$.
For this purpose, we first observe that from the second equation
in \eqref{cve}, we have
\begin{equation*}
\begin{split}
\|\u\|_{H^3}&\le
C\Big(\|G\|_{H^1}+\|\nabla\r^\gamma\|_{H^1}+\|\Dv(\r\F\F^\top)\|_{H^1}\Big)\\
&\le
C\Big(\|G\|_{H^1}+\|\nabla^2\r^\gamma\|_{L^2}+C+\|\nabla\F\|_{L^4}^2+\|\nabla\r\|_{L^4}^2+\|\nabla^2\F\|_{L^2}\Big)\\
&\le
C\Big(\|G\|_{H^1}+\|\nabla^2\r^\gamma\|_{L^2}+C+\|\nabla^2\F\|_{L^2}\Big),
\end{split}
\end{equation*}
since
$$\|\nabla\r\|_{L^4}^2\le \|\nabla\r\|_{L^2}^{\f12}\|\nabla\r\|_{L^6}^{\f32}\le C$$
and
$$\|\nabla\F\|_{L^4}^2\le \|\nabla\F\|_{L^2}^{\f12}\|\nabla\F\|_{L^6}^{\f32}\le C.$$

Applying $\nabla^2$ to the third equation of \eqref{cve} to yield
\begin{equation*}
\begin{split}
&(\nabla_{ij}\F)_t+\u\cdot\nabla(\nabla_{ij}\F)+\nabla_i\u\cdot\nabla\nabla_j\F+\nabla_{ij}\u\cdot\nabla\F\\
&\quad=(\nabla_{ij}\nabla\u)\F+\nabla\u\nabla_{ij}\F+\nabla_i\nabla\u\nabla_j\F+\nabla_j\nabla\u\nabla_i\F.
\end{split}
\end{equation*}
Multiplying the above identity by $2\nabla_{ij}\F$ and integrating
over $\R^3$, one obtains
\begin{equation*}
\begin{split}
&\f{d}{dt}\int_{\R^3}|\nabla^2\F|^2dx\\
&\le C\int_{\R^3}\Big(|\nabla^3\u||\F||\nabla^2\F|+|\nabla\u||\nabla^2\F|^2+|\nabla^2\u||\nabla\F||\nabla^2\F|\Big)dx\\
&\quad+\int_{\R^3}\u\cdot\nabla|\nabla^2\F|^2dx\\
&\le
C\int_{\R^3}\Big(|\nabla^3\u||\F||\nabla^2\F|+|\nabla\u||\nabla^2\F|^2+|\nabla^2\u||\nabla\F||\nabla^2\F|\Big)dx\\
&\quad-\int_{\R^3}\Dv\u|\nabla^2\F|^2dx\\
&\le C\int_{\R^3}\Big(|\nabla^3\u||\F||\nabla^2\F|+|\nabla\u||\nabla^2\F|^2+|\nabla^2\u||\nabla\F||\nabla^2\F|\Big)dx\\
&\le
C\Big(\|\nabla\u\|_{L^\infty}\|\nabla^2\F\|_{L^2}^2+\|\nabla^2\F\|_{L^2}^2+\|G\|_{H^1}^2+\|\nabla^2\r^\gamma\|_{L^2}^2+C+\|\nabla^2\u\|_{L^6}^2\Big)\\
&\le
C(\|\nabla\u\|_{L^\infty}+1)\|\nabla^2\F\|_{L^2}^2+C(\|G\|_{H^1}^2+\|\nabla^2\u\|_{L^6}^2+1),
\end{split}
\end{equation*}
since
$$\|\nabla\F\|_{L^3}\le \|\nabla\F\|_{L^2}^{\f12}\|\nabla\F\|_{L^6}^{\f12}\le C.$$
Because $\u\in L^2_tW^{2,6}_x$, by Gronwall's inequality, the above
inequality yields the bound
$$\|\F\|_{L^\infty_t H^2_x}\le C.$$
The estimates of $\F_t$ and $\F_{tt}$ follow from the similar energy
methods and the third equation of \eqref{cve}.
The proof is complete.
\end{proof}

When we state the $H^3$ regularity of the solution $(\r,\u,\F)$,
the following estimate is useful.

\begin{Lemma}\label{y2}
 For $0\le T\le T^*$, we have
$$\int_0^T\!\!\int_{\R^3}\r\u_{tt}^2dxdt+\sup_{0\le t\le
T}\int_{\R^3}|\nabla\u_t|^2dx\le C.$$
\end{Lemma}

\begin{proof}
Differentiating the second equation in \eqref{cve} with respect to
$t$, multiplying the resulting equation by $\u_{tt}$, and then
integrating over $\R^3$, one obtains
\begin{equation}\label{v22}
\begin{split}
&\int_{\R^3}\r\u_{tt}^2dx+\int_{\R^3}\r\u\cdot\nabla\u_t\cdot\u_{tt}dx+\f{d}{dt}\int_{\R^3}\left(\f{\mu}{2}|\nabla\u_t|^2+\f{\lambda+\mu}{2}|\Dv\u_t|^2\right)dx\\
& =A\int_{\R^3}(\r^\gamma)_t\Dv\u_{tt}dx-\int_{\R^3}(\r\F\F^\top)_t:\nabla\u_{tt}dx-\int_{\R^3}\r_t(\u_t+\u\cdot\nabla\u)\u_{tt}dx\\
&\qquad-\int_{\R^3}\r\u_t\cdot\nabla\u\cdot\u_{tt}dx\\
& =A\f{d}{dt}\int_{\R^3}(\r^\gamma)_t\Dv\u_{t}dx-A\int_{\R^3}(\r^\gamma)_{tt}\Dv\u_tdx-\f{d}{dt}\int_{\R^3}(\r\F\F^\top)_{t}:\nabla\u_{t}dx\\
&\qquad +\int_{\R^3}(\r\F\F^\top)_{tt}:\nabla\u_t dx
-\int_{\R^3}\r\u_t\cdot\nabla\u\cdot\u_{tt}dx-\f{d}{dt}\int_{\R^3}\f12\r_t|\u_t|^2dx\\
&\qquad+\f12\int_{\R^3}\r_{tt}|\u_t|^2dx-\f{d}{dt}\int_{\R^3}\r_t(\u\cdot\nabla\u)\u_tdx+\int_{\R^3}\r_{tt}\u\cdot\nabla\u\cdot\u_tdx\\
&\qquad
+\int_{\R^3}\r_t\u_t\cdot\nabla\u\cdot\u_tdx+\int_{\R^3}\r_t\u\cdot\nabla\u_t\cdot\u_tdx.
\end{split}
\end{equation}

Observe that
$$\left|\int_{\R^3}\r\u\cdot\nabla\u_t\cdot\u_{tt}dx\right|\le\i\|\sqrt{\r}\u_{tt}\|_{L^2}^2+C\|\nabla\u_t\|_{L^2}^2;$$
\begin{equation*}
\begin{split}
\left|\int_{\R^3}\r\u_t\cdot\nabla\u\cdot\u_{tt}dx\right|&\le
\i\|\sqrt{\r}\u_{tt}\|_{L^2}^2+C\|\sqrt{\r}\u_t\|_{L^3}^2\|\nabla\u\|_{L^6}^2\\
&\le
\i\|\sqrt{\r}\u_{tt}\|_{L^2}^2+C\|\sqrt{\r}\u_t\|_{L^2}\|\u_t\|_{L^6}\\
&\le \i\|\sqrt{\r}\u_{tt}\|_{L^2}^2+C\|\nabla\u_t\|_{L^2};
\end{split}
\end{equation*}
\begin{equation*}
\begin{split}
\left|\int_{\R^3}\r_{tt}|\u_t|^2dx\right|&=\left|\int_{\R^3}\Dv(\r_t\u+\r\u_t)|\u_t|^2dx\right|\\
&=\left|\int_{\R^3}(\r_t\u+\r\u_t)\cdot\nabla|\u_t|^2dx\right|\\
&\le
C\|\r_t\|_{L^6}\|\u\|_{L^6}\|\u_t\|_{L^6}\|\nabla\u_t\|_{L^2}+\int_{\R^3}\r|\u_t|^2|\nabla\u_t|dx\\
&\le
C\|\nabla\u_t\|_{L^2}^2+C\|\sqrt{\r}\u_t\|_{L^2}\|\sqrt{\r}\u_t\|_{L^6}\|\nabla\u_t\|_{L^3}\\
&\le
C\|\nabla\u_t\|_{L^2}^2+C\|\nabla\u_t\|_{L^2}\|\u_t\|_{H^2}\\
&\le
C\|\nabla\u_t\|_{L^2}^2+C\|\nabla\u_t\|_{L^2}(\|\sqrt{\r}\u_{tt}\|_{L^2}+\|\nabla\u_t\|_{L^2}+1)\\
&\le C\|\nabla\u_t\|_{L^2}^2+\i\|\sqrt{\r}\u_{tt}\|_{L^2}^2+C
\end{split}
\end{equation*}
since
$$\mu\D\u_t+(\lambda+\mu)\nabla\Dv\u_t=G_t+A\nabla(\r^\gamma)_t-\Dv(\r\F\F^\top)_t$$
implies that
\begin{equation*}
\begin{split}
\|\u_t\|_{H^2}&\le
C(\|G_t\|_{L^2}+\|\nabla(\r^\gamma)_t\|_{L^2}+\|\Dv(\r\F\F^\top)_t\|_{L^2})\\
&\le C(\|\sqrt{\r}\u_{tt}\|_{L^2}+\|\nabla\u_t\|_{L^2}+1);
\end{split}
\end{equation*}
 and
\begin{equation*}
\begin{split}
\left|\int_{\R^3}\r_{tt}\u\cdot\nabla\u\cdot\u_{t}dx\right|&\le
\|\r_{tt}\|_{L^2}\|\u\cdot\nabla\u\|_{L^3}\|\u_t\|_{L^6}\\
&\le C\|\r_{tt}\|_{L^2}^2+C\|\nabla\u_t\|_{L^2}^2;
\end{split}
\end{equation*}
\begin{equation*}
\begin{split}
\left|\int_{\R^3}\r_t\u_t\cdot\nabla\u\cdot\u_tdx\right|&\le
\|\r_t\|_{L^2}\|\u_t^2\|_{L^3}\|\nabla\u\|_{L^6}\le
C\|\nabla\u_t\|_{L^2}^2;
\end{split}
\end{equation*}
\begin{equation*}
\begin{split}
\left|\int_{\R^3}\r_t\u\cdot\nabla\u_t\cdot\u_tdx\right|&\le
\|\r_t\|_{L^3}\|\u\|_{L^\infty}\|\nabla\u_t\|_{L^2}\|\u_t\|_{L^6}\\
&\le C\|\nabla\u_t\|_{L^2}\|\u_t\|_{L^6}\le
C\|\nabla\u_t\|_{L^2}^2.
\end{split}
\end{equation*}

Substituting the above estimates  back to \eqref{v22} and then
integrating over $[0,t]$ with $0\le t\le T$, one has
\begin{equation}\label{v23b}
\begin{split}
&\int_0^t\int_{\R^3}\r\u_{tt}^2dxds+\f{\mu}{4}\sup_{0\le t \le T}\int_{\R^3}|\nabla\u_t|^2dx\\
&\le C\i\int_0^T\!\!\int_{\R^3}\r\u_{tt}^2dxds +C\int_0^T\!\!\int_{\R^3}\|\nabla\u_t|^2dxds+C.
\end{split}
\end{equation}
Here we used the following estimates: for all $0\le t\le T$,
\begin{equation*}
\begin{split}
\left|\int_{\R^3}(\r^\gamma)_t\Dv\u_{t}dx\right|
&\le\|(\r^\gamma)_t\|_{L^2}\|\nabla\u_t\|_{L^2}\\
&\le C+\f{\mu}{8}\sup_{0\le t\le T}\int_{\R^3}|\nabla\u_t|^2dx;
\end{split}
\end{equation*}
\begin{equation*}
\begin{split}
\left|\int_{\R^3}(\r\F\F^\top)_{t}:\nabla\u_{t}dx\right|
&\le (\|\r_t\|_{L^2}+\|\F_t\|_{L^2})\|\nabla\u_t\|_{L^2}\\
&\le C+\f{\mu}{8}\sup_{0\le t\le T}\int_{\R^3}|\nabla\u_t|^2dx;
\end{split}
\end{equation*}
\begin{equation*}
\begin{split}
\left|\int_{\R^3}\r_t|\u_t|^2dx\right|&=\left|\int_{\R^3}\Dv(\r\u)\u_t^2dx\right|=\left|\int_{\R^3}\r\u\nabla\u_t^2dx\right|\\
&\le \|\sqrt{\r}\u_t\|_{L^2}\|\nabla\u_t\|_{L^2}\le
C+\f{\mu}{8}\sup_{0\le t\le T}\int_{\R^3}|\nabla\u_t|^2dx;
\end{split}
\end{equation*}
and
\begin{equation*}
\begin{split}
\left|\int_{\R^3}\r_t(\u\cdot\nabla\u)\u_tdx\right|
&\le \|\r_t\|_{L^6}\|\u\|_{L^6}\|\nabla\u\|_{L^6}\|\nabla\u_t\|_{L^2}\\
&\le C+\f{\mu}{8}\sup_{0\le t\le T}\int_{\R^3}|\nabla\u_t|^2dx.
\end{split}
\end{equation*}

Choosing a sufficient small $\i$, then \eqref{v24}, combining with
Gronwall's inequality, gives
$$\int_0^T\!\!\int_{\R^3}\r\u_{tt}^2dxdt+\sup_{0\le t\le
T}\int_{\R^3}|\nabla\u_t|^2dx\le C.$$
The proof is complete.
\end{proof}

With aid of the above estimate, we can obtain the $H^3$ regularity as
follows.

\begin{Lemma}
$$\|\r\|_{L^\infty_tH^3_x}+\|\F\|_{L^\infty_tH^3_x}+\|\u\|_{L^\infty_tH^3_x}\le C.$$
\end{Lemma}

\begin{proof}
It follows from Lemma \ref{y1} and Lemma \ref{y2} that
$$G\in L^\infty_t H^1_x,\quad \nabla\r^\gamma\in L^\infty_t H^1_x,\quad
\Dv(\r\F\F^\top)\in L^\infty_t H^1_x,$$ 
which gives
$$\mu\D\u+(\lambda+\mu)\nabla\Dv\u=G+A\nabla\r^\gamma-\Dv(\r\F\F^\top)\in L^\infty_t H^1_x.$$ 
From the standard estimate for elliptic equations, we have
$$\|\u\|_{L^\infty_t H^3_x}\le C.$$

From the second equation in \eqref{cve}, Lemma \ref{y1} and Lemma
\ref{y2}, we have
$$\mu\D\u_t+(\lambda+\mu)\nabla\Dv\u_t=G_t+A\nabla(\r^\gamma)_t-\Dv(\r\F\F^\top)_t\in L^2((0,T)\times \R^3).$$
Therefore, again from the standard estimate for elliptic
equations, we have
$$\u_t\in L^2_t H^2_x,\quad G\in L^2_t H^2_x.$$

With those bounds, we can apply a similar argument in Lemma
\ref{y1} to deduce from the first and the third equation in
\eqref{cve} that
$$\|\r\|_{L^\infty_t H^3_x}+\|\F\|_{L^\infty_t H^3_x}\le C.$$
The proof is complete.
\end{proof}

Now we are able to state the proof of Theorem \ref{vMT}.

\begin{proof}[Proof of Theorem \ref{vMT}]
From the previous lemmas, the functions $(\r, \F, \u)|_{t=T^*}$
satisfy the same conditions as the initial data. Moreover
$$G=\r\u_t+\r\u\cdot\nabla\u\in L^\infty_t H^1_x$$
and
$$(-\mu\D\u-(\lambda+\mu)\nabla\Dv\u+A\nabla\r^\gamma)|_{t=T^*}=\r (T^*)g$$
with $$g\in H^1(\R^3), \quad\text{and}\quad \sqrt{\r(T^*)}g\in L^2.$$ Therefore, we
can take $(\r,\F,\u)(T^*)$ as the initial data and apply the local
existence theorem to extend our local classical solution beyond
$T^*$, and this contradicts the assumption on $T^*$.
The proof is complete.
\end{proof}

\bigskip\bigskip

\section*{Acknowledgments}

X. Hu's research was supported in part by the National Science Foundation.
D. Wang's research was supported in part by the National Science
Foundation and  the Office of Naval Research. 

\end{document}